\documentclass[11pt]{amsart}
\numberwithin{equation}{section}

\usepackage{amsmath,amsfonts,amsthm,amsopn,color,amssymb,enumitem}
\usepackage{amsthm,amsmath,amssymb,amscd}
\usepackage{amssymb}
\usepackage{graphics}
\usepackage{color}
\usepackage{esint}

\newtheorem{Theorem}{Theorem}[section]      
\newtheorem{Proposition}[Theorem]{Proposition}    
\newtheorem{Lemma}[Theorem]{Lemma}            
\newtheorem{Remark}[Theorem]{Remark}           

\def\R{\mathbb{R}}
\def\RN{\mathbb{R}^N}
\def\Sp{\mathbb{S}^{N-1}}
\def\S{\mathbb{S}^{N-1}}
\def\N{\mathbb{N}}

\def\e{\varepsilon}
\def\pp{\partial}
\def\weakto{\rightharpoonup}

\newcommand{\rmnote}[1]{}

\title[Overdetermined elliptic problems in exterior domains]{Solutions to overdetermined elliptic problems in nontrivial exterior domains}



\begin{document}

\author{Antonio Ros}
\address{(A.~Ros)
Departamento de Geometr\'ia y Topolog\'ia,
Universidad de Granada,
Campus Fuentenueva,
18071 Granada,
Spain}
\email{aros@ugr.es}

\author{David Ruiz}
\address{(D.~Ruiz)
Departamento de An\'alisis matem\'atico, Universidad de Granada,
Campus Fuente-nueva, 18071 Granada, Spain} \email{daruiz@ugr.es}

\author{Pieralberto Sicbaldi}
\address{(P.~Sicbaldi)
Aix Marseille Universit\'e - CNRS, Centrale Marseille - I2M, Marseille, France}
\email{pieralberto.sicbaldi@univ-amu.fr}

\maketitle

\begin{abstract} In this paper we construct nontrivial exterior domains $\Omega \subset \R^N$, for all $N\geq 2$, such that the problem
$$\left\{\begin{array} {ll}
-\Delta u +u -u^p=0,\ u >0 & \mbox{in }\; \Omega, \\[1mm]
               \ u= 0 & \mbox{on }\; \partial \Omega, \\[1mm]
               \ \frac{\partial u}{\partial \nu} = \mbox{cte} & \mbox{on }\; \partial \Omega,
\end{array}\right.$$
admits a positive bounded solution. This result gives a negative answer to the Berestycki-Caffarelli-Nirenberg conjecture on overdetermined elliptic problems in dimension 2, the only dimension in which the conjecture was still open. For higher dimensions, different counterexamples have been found in the literature; however, our example is the first one in the form of an exterior domain.
\end{abstract}

\section{Introduction}

This paper is concerned with the existence of solutions of a semilinear overdetermined elliptic problem in the form
\begin{equation}\label{pr_bis}
\left\{\begin{array} {ll}
\Delta u + f(u) = 0 & \mbox{in }\; \Omega,\\
                        u > 0 & \mbox{in }\; \Omega,\\
               u= 0 & \mbox{on }\; \pp \Omega, \\
\frac{\partial u}{\partial \nu}=c \neq 0 &\mbox{on }\; \pp
\Omega.\,
\end{array}\right.
\end{equation}
Here $\Omega \subset \R^N$ is a regular domain, $f$ is a Lipschitz function and $\nu$ stands for the exterior normal vector to
$\partial \Omega$. Observe that the presence of two boundary conditions makes the problem overdetermined. Overdetermined boundary conditions arise naturally in free boundary problems, when the variational structure imposes
suitable conditions on the separation interface, see for example
\cite{Alt-Caf}.

\medskip

In 1971 J. Serrin proved that if \eqref{pr_bis} is solvable for a bounded domain $\Omega$, then $\Omega$ must be a ball (\cite{Serrin, Puc-Ser}). This is also true if we replace the Laplacian operator by another strongly elliptic operator and if the function $f$ depends also on the gradient of $u$. This result has many applications in Physics, and some of them are the following: 1) when a viscous incompressible fluid is moving in straight parallel streamlines through a pipe of given cross section, the tangential stress per unit area on the pipe wall is the same at all points if and only if the cross section is circular; 2) when a solid straight bar is subject to torsion, the magnitude of the resulting traction which occurs at the surface of the bar is independent of the position if and only if the bar has a circular cross section; 3) when a liquid is rising in a straight capillary tube, the liquid will rise to the same height at the tube wall if and only if the tube has circular section.
Serrin's proof is based on the {\it Alexandrov reflection principle}, introduced in 1956 by Alexandrov in \cite{Alex} to prove that the only compact, connected, embedded hypersurfaces in $\R^N$ whose mean curvature is constant are the spheres.
The reflection principle was used also in 1979 by Gidas, Ni and Nirenberg \cite{GNN} to derive radial symmetry results for positive solutions of semilinear elliptic equations. After that paper the reflection principle has been named the {\it moving plane method}.

\medskip

A natural dual version of the previous situation is to consider problem \eqref{pr_bis} in exterior domains, i.e., domains $\Omega$ given as the complement of a compact connected region $D \subset \R^N$. In Physics this situation corresponds to the case of very big domains (mathematically considered as unbounded) with a hole. We refer the reader to the survey \cite{sirakov_survey} for more specific applications in Physics of elliptic overdetermined problems in exterior domains.

\medskip

In the case of exterior domains, the problem that has been considered is the following:
\begin{equation}\label{pr_ext}
\left\{\begin{array} {ll}
\Delta v + g(v) = 0 & \mbox{in }\; \Omega,\\
               v= a >0 & \mbox{on }\; \pp \Omega, \\
\frac{\partial v}{\partial \vec{\nu}}= c&\mbox{on }\; \pp
\Omega,\\
0 \leq v < a & \mbox{in }\; \Omega.\\
\end{array}\right.
\end{equation}
With the change of variables $u:=a-v$ we have immediately a problem in the form \eqref{pr_bis} with the extra assumptions $u \leq a$. In this framework, the main research line has aimed to prove the counterpart of the Serrin's symmetry result, that is, to prove that $\Omega$ is the complement of a ball. For example under the assumptions that $g(t) \geq 0$ and that $t^{-\frac{n+2}{n-2}}g(t)$ is nonincreasing, Aftalion and Busca proved in \cite{aft-bus} that if problem \eqref{pr_ext} has a solution then $\Omega$ is the complement of a ball. In \cite{reichel} Reichel proved the same symmetry result but under different assumptions: he assumes that $g(t)$ is decreasing for small positive $t$ and that $u \to 0$ at infinity.
This last result in \cite{reichel} is still true if we replace the Laplacian operator $\Delta$ by a regular strongly elliptic operator, as shown by Sirakov in \cite{sir}. In the proofs the main tool used is the moving plane method from infinity, sometimes combined with the moving spheres method. As a consequence, their proofs show not only symmetry, but also monotonicity along the radius. To fix ideas, if $\Omega$ is a exterior domain and $f(u)=u-u^3$ (the so-called Allen-Cahn nonlinearity), one infers from \cite{reichel} that \eqref{pr_bis} is solvable only if $\Omega$ is the complement of a ball.
\medskip

Our first observation is that there are radially symmetric solutions of problem \eqref{pr_bis} which are not monotone along the radius. For instance, there exists a non-monotone radial solution to the problem:
\begin{equation}\label{NLS}
\left\{\begin{array} {ll}
\Delta u + u^p-u = 0 & \mbox{in }\; B_R^c,\\
                        u > 0 & \mbox{in }\;  B_R^c,\\
               u= 0 & \mbox{on }\; \pp B_R,
\end{array}\right.
\end{equation}
for any $p>1$ and $R>0$, where $B_R$ is the ball of radius $R$ and $B_R^c$ is its complement (see for instance \cite{esteban}). This equation receives the name of Nonlinear Schr\"{o}dinger Equation and has been widely studied in the literature. Its solution increases in the radius up to a certain maximum, and then it decreases and converges to $0$ at infinity.
Therefore, the usage of the moving plane method from infinity is intrinsically restricted to some kind of nonlinearities and/or solutions $u$. The goal of this paper is to prove that \eqref{pr_bis} is solvable for some exterior domain different from the complement of a ball. For that, we use a local bifurcation argument from the solutions of \eqref{NLS}.

\medskip

Before going further presenting our results, let us review the literature on overdetermined semilinear elliptic problems.
In \cite{BCN} Berestycki, Caffarelli and Nirenberg consider free boundary problems where the variational structure imposes overdetermined conditions on the boundary. The study of the regularity of the solutions by a blow-up technique led them to study problem \eqref{pr_bis} in epigraphs. Under some hypothesis on the nonlinearity $f$ and on the behavior of the epigraph at infinity, they proved that the epigraph must be a half-space (these results were later extended by Farina and Valdinoci \cite{FV}). Motivated by this, and by the aforementioned results on exterior domains by Aftalion, Busca and Reichel, they proposed in \cite{BCN} the following conjecture:

\medskip

{\bf BCN Conjecture}: If $\mathbb{R}^N \backslash
\overline{\Omega}$ is connected, then the existence of a bounded
solution to problem \eqref{pr_bis} implies that $\Omega$ is either
a ball, a half-space, a generalized cylinder $B^k \times
\mathbb{R}^{N-k}$ ($B^k$ is a ball in $\mathbb{R}^k$), or the
complement of one of them.

\medskip

This conjecture has been answered negatively for $N\geq 3$ in \cite{Sicbaldi}, where the third author finds a periodic perturbation of
the straight cylinder $B^{N-1} \times \mathbb{R}$ that supports a
periodic solution to problem (\ref{pr_bis}) with $f(t) = \lambda\,
t$.

\medskip

In the last years, a parallelism between overdetermined
elliptic problems and constant mean curvature surfaces, in the spirit of the correspondence of Alexandrof's and Serrin's results, has been observed. Indeed, the counterexample to the BCN Conjecture built in \cite{Sicbaldi} belongs
to a smooth one-parameter family that can be seen as a counterpart
of the family of {\it Delaunay surfaces}, see \cite{Sch-Sic}. Such domains exist also in other homogeneous manifolds, as $\mathbb{S}^n \times \R$ or $\mathbb{H}^n \times \R$, as shown in \cite{Mor-Sic} as the counterpart of other well known families of constant mean curvature surfaces.
In \cite{HHP} H\'elein, Hauswirth and Pacard establish a kind of Weierstrass representation for overdetermined elliptic problems in dimension 2 with $f \equiv 0$ in analogy with minimal surfaces. Moreover, Traizet finds a one-to-one correspondence between solutions of problem \eqref{pr_bis} in dimension $2$ with $f \equiv 0$ and a special class of minimal surfaces (\cite{traizet}). Finally,
in \cite{DPW} Del Pino, Pacard and Wei consider problem
\eqref{pr_bis} for functions $f$ of Allen-Cahn type and they build
new solutions in domains $\Omega$  with boundary close to a dilated embedded minimal surface in $\R^3$ with finite total curvature and nondegenerate, or to a dilated Delaunay surface.

\medskip

If $\Omega$ is an epigraph, the problem is also related to the {\it De Giorgi's conjecture} (1978), that is still open in its full generality. This conjecture states that the entire solutions of the Allen-Cahn equation $\Delta u + u - u^3 =0$ monotone in one direction must have level sets which are
parallel hyperplanes if $N \leq 8$. The relationship between the De Giorgi's conjecture and overdetermined problems is not surprising if we recall that this conjecture is the counterpart of the {\it Bernstein's conjecture} on minimal surfaces (1914), that stated that all entire minimal graphs in $\R^N$ should be hyperplanes, and which has been disproved by E. Bombieri, E. De Giorgi and E. Giusti for $N\geq9$ (\cite{BDG}). Starting from the Bombieri-De Giorgi-Giusti entire minimal graph, Del Pino, Kowalczyk and Wei build entire nontrivial monotone solutions to the Allen-Cahn equation if $N\geq 9$. In this spirit, Del Pino, Pacard and Wei has recently built nontrivial solutions for \eqref{pr_bis} for $f$ of Allen-Cahn type in nontrivial epigraphs if $N \geq 9$, see \cite{DPW}. In \cite{WW} Wang and Wei prove that this type of solutions do not exist if $N \leq 8$, a result that can be put in analogy with that of Savin for the De Giorgi conjecture (\cite{savin}). Finally, the notion of stability plays an important role in the De Giorgi conjecture, and also in overdetermined problems, see \cite{wang}.

\medskip

Coming back to the BCN Conjecture, we point out that all counterexamples mentioned above require $N \geq 3$, and we underline that all the examples of domains solving an overdetermined elliptic problem are linked to minimal or constant mean curvature surfaces.

\medskip

In this paper we give a counterexample in the form of a exterior domain for any dimension $N \geq 2$. This gives a definitive negative answer to the conjecture. Partial positive answers to the BCN conjecture in dimension 2 have been given in several works, see \cite{FV, HHP, nosotros, Ros-Sic, traizet, WW}. In \cite{nosotros} the authors show that the conjecture holds in dimension 2 under the hypothesis that $\partial \Omega$ is unbounded. The counterexample we give in this paper shows that such hypothesis is actually sharp.

\medskip

Finally, this is the first example of a domain solving an overdetermined elliptic problem that has no clear counterpart in the theory of minimal or constant mean curvature surfaces.

\medskip A first statement of our result (see Section 2 for a more detailed statement) is the following:

\begin{Theorem} \label{teo}
Let $N \in \N$, $N\geq 2$, $1<p< \frac{N+2}{N-2}$ ($p>1$ if $N=2$). There exist smooth exterior domains $\Omega$ different from the complement of a ball such that the overdetermined problem
\begin{equation} \label{ovdet}
\left\{\begin{array} {ll}
-\Delta u +u -u^p=0,\ u >0 & \mbox{in }\; \Omega, \\
               \ u= 0 & \mbox{on }\; \partial \Omega, \\
               \ \frac{\partial u}{\partial \nu} = \mbox{cte} & \mbox{on }\; \partial \Omega,
\end{array}\right.
\end{equation}
admits a bounded solution.
\end{Theorem}

\medskip

Observe that, for any $R>0$, the solutions to \eqref{NLS} form a trivial family of solutions of \eqref{ovdet}. In this paper we use a local bifurcation argument to show that, from this family of trivial solutions, there are nontrivial solutions in nontrivial domains bifurcating at some values of the radius. The proof uses a general bifurcation result in the spirit of Krasnoselskii.
In order to do that, the nondegeneracy of the Dirichlet problem is essential, but in general this is false at least for some radii $R$. Under some symmetry assumptions, we find a spectral gap for the Dirichlet problem associated to \eqref{NLS}, that is, we show that it is nondegenerate for $R \in (0, R_0)$, for some $R_0>0$. Another important issue of the proof is to show that bifurcation occurs exactly in that interval. This is made by studying the behavior of the first Steklov eigenvalue of the linearized operator.

\medskip

The paper is organized as follows. In Section 2 we present some notations and preliminaries, we state our precise result, and we show the existence of the spectral gap for the Dirichlet problem; some proofs of the results of this section are postponed to the last section. In Section 3 we define the operator that appears naturally in our problems, and we compute its linearization. Section 4 is devoted to the study of the linearized operator and its spectrum. Finally, in Section 5 we use a local bifurcation result to conclude the proof.

\medskip

{\bf Acknowledgments. } A. Ros has been partially supported by Mineco-Feder Grant MTM2014-52368-P. D.
Ruiz has been supported by the Mineco-Feder Grant MTM2015-68210-P and by J.
Andalucia (FQM 116). P. Sicbaldi was partially supported ANR-11-IS01-0002 Grant.

\section{Preliminaries}

Let us first set some notations. Given a symmetry group $G$ acting on $\R^N$, we say that $\Omega \subset \R^N$ is $G$-symmetric if it is invariant under the action of the group $G$. In such case, we can define the Sobolev spaces of $G$-symmetric functions as follows:
$$ H^1_G(\Omega)= \{u \in H^1(\Omega):\ u= u \circ g \ \forall g \in G \}, $$
$$ H^1_{0,G}(\Omega)= \{u \in H_0^1(\Omega):\ u= u \circ g \ \forall g \in G \}, $$
and by $H^{-1}_{G}(\Omega)$ the dual space of $H^1_{0,G}(\Omega).$
We will use the same kind of notations for functions defined in $\partial \Omega$. In particular:
$$ H^{1/2}_G(\partial \Omega)= \{u \in H^{1/2}(\partial \Omega):\ u= u \circ g \ \forall g \in G \}. $$
We denote by $B_R \subset \R^N$ the ball of radius $R$ centered at $0$, and we may also write $\Sp$
instead of $\partial B_1$.
If $\Omega$ is radially symmetric, we shall denote the spaces of radially symmetric functions as:
$$ H^1_r(\Omega)= \{u \in H^1(\Omega):\ u(x)= u(|x|) \  a.e. \ x \in \Omega \}, $$
$$ H^1_{0,r}(\Omega)= \{u \in H_0^1(\Omega):\ u(x)= u(|x|) \  a.e. \ x \in \Omega \}. $$
For a function $u \in H^1(\Omega)$, we denote $\| u \| = \Big ( \| \nabla u \|_{L^2}^2 + \| u \|_{L^2}^2 \Big)^{1/2}$ its Sobolev norm. Other norms will be clear from the subscript.
In the case of the Holder regularity we can define the following spaces:
$$C^{k,\alpha}_G(\Omega)= \{u \in C^{k,\alpha}(\Omega):\ u= u \circ g \ \forall g \in G \}, $$
$$C^{k,\alpha}_G(\pp\Omega)= \{u \in C^{k,\alpha}(\pp\Omega):\ u= u \circ g \ \forall g \in G \}. $$

Moreover, we will denote by $C^{k,\alpha}_{G,m}(\Sp)$ the set of functions in $C^{k,\alpha}_G(\Sp)$ whose mean is $0$.
Given a positive function $w \in C^{2,\alpha}_G(\Sp)$ let us denote
\[
B_{w} : =  \left\{   x \in \mathbb R^{N} \,\,\ : \,\,\, 0 \leq |x|   < w\left( \frac{x}{|x|}\right)\right\} \, .
\]
and $B_w^c$ its complement in $\R^N$.

\medskip We denote $\Delta_{\Sp}$ the Laplace-Beltrami operator in $\Sp$, and $\{\mu_{i}\}_{i\in\N}$ its eigenvalues, i.e. $\mu_i=i(i+N-2)$. From now on, we will fix a symmetry group $G$ with the following property:
\begin{enumerate}[label=(G), ref=(G)]
\item \label{G} $G$ leaves invariant the origin and, denoting by $\{\mu_{i_k}\}_{k\in\N}$ the eigenvalues of $\Delta_{\Sp}$ restricted to $G-$symmetric functions and $m_k$ their multiplicity, we require $i_1 \geq 2$ and $m_1$ odd.
\end{enumerate}

We are now able to state the main result of this paper, from which Theorem \ref{teo} follows immediately.

\begin{Theorem} \label{main} Let $N \in \N$, $N\geq 2$, $1<p< \frac{N+2}{N-2}$ ($p>1$ if $N=2$). Let $G$ be a group of symmetries of $\RN$ satisfying \ref{G}. Then there exist $R_*=R_*(i_1, p)>0$, a sequence of non-zero functions $v_n \in C^{2,\alpha}_{G,m}(\Sp)$ converging to $0$, and a sequence of positive real numbers $R_n$ converging to $R_*$ such that the overdetermined problem:
\[
\left\{\begin{array} {ll}
-\Delta u +u -u^p=0, & \mbox{in }\; B_{R_n(1+v_n)}^c \\[2mm]
               \ u= 0 & \mbox{on }\; \pp B_{R_n(1+v_n)} \\[2mm]
               \ \frac{\partial u}{\partial \nu} = \mbox{cte} & \mbox{on }\; \pp B_{R_n(1+v_n)}
\end{array}\right.
\]
admits a positive bounded solution $u \in C^{2,\alpha}_G\left(B_{R_n(1+v_n)}^c\right) \cap H^1_{0,G}\left(B_{R_n(1+v_n)}^c\right)$.
\end{Theorem}

\begin{Remark}\rm{
There are many examples of groups $G$ satisfying \ref{G}. For instance, if $1\leq m \leq N-1$, the group $G=O(m) \times O(N-m)$ satisfies that $i_1=2$ and $m_1=1$. Indeed, in this case, the corresponding eigenvalue is given as the restriction to $\S$ of the 2-homogeneous harmonic polynomial: $$p(x)= (N-m)(x_1^2 + \dots + x_m^2) - m (x_{m+1}^2 + \dots + x_N^2).$$
In dimension $2$, we can take as $G$ any dihedral group $\mathbb{D}_k$, $k\geq 3$. In this case, $i_1= k$ and $m_1=1$.
In dimension $3$ we can take $G$ as the group of isometries of the tetrahedron ($i_1=3$ and $m_1=1$), the octahedron ($i_1 = 4$ and $m_1=1$) or the icosahedron ($i_1 = 6$ and $m_1=1$), see \cite{laporte}.}

\end{Remark}

\begin{Remark}\rm{
One could ask whether two different groups $G_1$, $G_2$ give rise to different domains $\Omega$. The answer is (partially) affirmative. Indeed, define $G= \langle G_1 \,,\, G_2\rangle$, and denote:
\begin{enumerate}
\item $\{\mu_{i_k}\}$ the eigenvalues of $\Delta_{\Sp}$ restricted to $G_1$-symmetric functions,
\item $\{\mu_{j_k}\}$ the eigenvalues of $\Delta_{\Sp}$ restricted to $G_2$-symmetric functions,
\item $\{\mu_{l_k}\}$ the eigenvalues of $\Delta_{\Sp}$ restricted to $G$-symmetric functions.
\end{enumerate}
Clearly, $l_1 \geq \max \{ i_1, j_1\}$. If $l_1 > \min \{ i_1, j_1 \}$, then the two groups $G_1$ and $G_2$ give rise to different domains $\Omega$. In particular, this is true if $i_1 \neq j_1$. In fact the value of the bifurcation radius $R_*$ is different; this is due to the fact that the value $R_*$ is strictly increasing with respect to $i_1$, as can be see in the proof of Lemma \ref{caracola}.}
\end{Remark}

As commented in the introduction, we will prove Theorem \ref{main} by means of a bifurcation argument to show the existence of such domain $\Omega$ close to the exterior of a ball. For that, we shall need some facts of the Dirichlet problem; given any $p>1$, consider:
\begin{equation} \label{dirichlet}
\left\{\begin{array} {ll}
-\Delta u +u -u^p=0,\ u >0 & \mbox{in }\; B_R^c, \\
               \ u= 0 & \mbox{on }\; \pp B_R. \\
\end{array}\right.
\end{equation}
It will be convenient to make a change of scale and pass to the equivalent problem:
\begin{equation} \label{dirichlet2}
\left\{\begin{array} {ll}
- \lambda \Delta u +u - u^p=0,\ u >0   & \mbox{in }\; B_1^c, \\
               \ u= 0 & \mbox{on }\; \pp B_1, \\
\end{array}\right.
\end{equation}
where $\lambda = \frac{1}{R^2}$. In the proposition below we list some known properties of this problem.

\begin{Proposition} \label{list} There follows:
\begin{enumerate}
\item[a)] For any $\lambda>0$, there exists a radially symmetric $C^{\infty}$ solution of \eqref{dirichlet2}. This solution increases in the radius up to a certain maximum, and then it decreases and converges to $0$ at infinity.
\item[b)] The positive and radial solution to \eqref{dirichlet2} is unique: we denote it by $u_\lambda$.
\item[c)] Let us define the linearized operator $L_\lambda: H^1_{0,G}(B_1^c) \to H^{-1}_{G}(B_1^c)$,
\begin{equation}\label{linear-d}
L_\lambda(\phi) = - \lambda \Delta \phi + \phi  - p u_\lambda^{p-1} \phi \,,
\end{equation}
and consider the eigenvalue problem:
$$ L_\lambda(\phi) = \tau\, \phi.$$
In the space of radially symmetric functions $H^1_{0,r}(B_R^c)$ this problem has a unique negative eigenvalue and no zero eigenvalues. In other words, $u_\lambda$ is nondegenerate in $H^1_{0,r}(B_1^c)$ and has Morse index 1. We denote by $z_\lambda \in H^1_{0,r}(B_1^c)$ (normalized by $\| z_\lambda \|=1$) the positive eigenfunction with negative eigenvalue, i.e.
\begin{equation} \label{z}
\left\{\begin{array} {ll}
-\Delta z_\lambda + z_\lambda  - p u_\lambda^{p-1} z_\lambda = \tau_0 z_\lambda & \mbox{in }\; B_1^c, \\
         \      z_\lambda= 0 & \mbox{on }\; \pp B_1, \\
\end{array}\right.
\end{equation}
where $\tau_0= \tau_0(\lambda)<0$. Moreover $z_\lambda$ is a $C^{\infty}$ function.

\end{enumerate}
\end{Proposition}

\begin{proof}
Statement a) is quite well known and has been proved in \cite{esteban}, for instance. The results b) and c) are more recent and have been obtained in \cite{f-m-tanaka, tang}.
\end{proof}

Let us define the bilinear operator associated to \eqref{linear-d}: $Q_\lambda: H^1_{0,G}(B_1^c) \times H^1_{0,G}(B_1^c) \to \R$,
\begin{equation} \label{bilinear0}  Q_\lambda(\psi_1, \psi_2)=  \int_{B_1^c} \lambda \nabla \psi_1 \cdot \nabla \psi_2 + \psi_1 \, \psi_2 - p u_\lambda^{p-1} \psi_1 \, \psi_2\,. \end{equation}
By Proposition \ref{list}, c), $Q_\lambda$ is positive definite for $\psi \in H_{0,r}^1(B_1^c)$ with $\int_{B_1^c} \psi \, z_{\lambda} =0$. In next lemma we show that this property may fail if we do not impose radial symmetry. This might be known in the literature, but we have not been able to find a specific reference.
\begin{Proposition} \label{G-deg} Let $G$ be a symmetry group satisfying hypothesis \ref{G}. Then there exists $\e > 0$ such that for any $\lambda \in (0, \e)$, there exists $\psi \in H_{0,G}^1(B_1^c)$ such that
\begin{enumerate}
\item $\int_{B_1^c} \psi z_{\lambda}=0$.
\item $Q_\lambda(\psi, \psi)< 0$.
\end{enumerate} \end{Proposition}


\begin{Remark}\rm{
The proof of Proposition \ref{G-deg} can be adapted to show that the Morse index of $u_\lambda$ in $H_{0}^1(B_1^c)$ diverges as $\lambda \to 0$. This is in contrast with what happens in the radial case. Hence, one expects the existence of infinitely many branches of nonradial solutions to the problem \eqref{dirichlet2} bifurcating from $u_\lambda$. As far as we know, this result has not been explicitly written in the literature. In any case, the existence of this kind of solutions is outside the scope of this paper.}
\end{Remark}

\begin{proof}[Proof of Proposition \ref{G-deg}]

Let $\mu_{i_1}$ be the first eigenvalue of the operator $\Delta_{\Sp}$ restricted to $G$-symmetric functions, and $\phi$ one of the corresponding normalized eigenfunctions.
Let us define the function $\psi$ in polar coordinates: $\psi (r, \theta)= u_\lambda(r) \phi(\theta)$, with $\theta \in \Sp$. Since $z_\lambda$ is radially symmetric, $\int_{B_1^c} \psi z_{\lambda}=0$. Observe now that
$$ | \nabla \psi |^2 = u_\lambda'(r)^2\, \phi^2(\theta) + \frac{1}{r^2} \, u_\lambda(r)^2\,   |\nabla_{\theta} \phi(\theta)|^2.$$
Therefore,
\begin{align*} Q_\lambda(\psi, \psi) =  \int_1^{+\infty} r^{N-1} \left(\lambda u_\lambda'(r)^2 + u_{\lambda}(r)^2 - p |u_{\lambda}(r)|^{p+1} \right) \int_{\Sp} \phi^2(\theta) \, d \theta\, dr   \\ + \lambda \int_1^{+\infty} r^{N-1} \, \frac{1}{r^2} \,u_\lambda(r)^2 \int_{\Sp}|\nabla_{\theta} \phi (\theta)|^2 \, d \theta\, dr \\ =  \int_1^{+\infty}  \left(\lambda u_\lambda'(r)^2 + u_{\lambda}(r)^2 - p |u_{\lambda}(r)|^{p+1}  + \lambda \frac{\mu_{i_1}}{r^2} u_\lambda(r)^2 \right)r^{N-1} \, dr\,. \end{align*}
If we multiply the equation in \eqref{dirichlet2} (with $u=u_\lambda$) by $u_\lambda$ and integrate, we obtain:
$$ \int_1^{+\infty}  \left(\lambda u_\lambda'(r)^2 + u_{\lambda}(r)^2 \right) r^{N-1}\, dr = \int_1^{+\infty}   |u_{\lambda}(r)|^{p+1}\,  r^{N-1} \, dr.$$
Plugging this identity in the above expression, we obtain:
\begin{align*} Q_\lambda(\psi, \psi) =  \int_1^{+\infty} \left( (1-p) \left( \lambda u_\lambda'(r)^2 + u_{\lambda}(r)^2\right) + \lambda\, \frac{ \mu_{i_1}}{r^2} \, u_\lambda(r)^2 \right)r^{N-1} \, dr \\  \leq \int_1^{+\infty} \left(1-p+ \lambda \mu_{i_1}\right) u_{\lambda}(r)^2\,  r^{N-1} \, dr, \end{align*}
and this last quantity is negative if $\lambda < \e := \frac{p-1}{ \mu_{i_1}}$.
\end{proof}

%
%
%
%
%
%

In view of Proposition \ref{G-deg}, the operator $L_\lambda$ may be degenerate if we consider non radially symmetric functions. However, as a consequence of next proposition, we conclude that $L_\lambda$ is nondegenerate for large values of $\lambda$.

\begin{Proposition} \label{key1} There exists $M>0$ such that for any $\lambda > M$,  $Q_\lambda(\psi, \psi)>0 $ for any $\psi \in H^1_{0,G}(B_1^c)$ such that $\int_{B_1^c} \psi z_{\lambda} =0$. \end{Proposition}



The proof of this proposition is postponed to Section \ref{ultima}.

\medskip

We define:
\begin{equation} \label{Lambda0}  \Lambda_0 = \sup \left\{ \lambda >0: Q_\lambda(\psi, \psi) \leq 0 \mbox{ for some } \psi \in H_{0,G}^1(B_1^c)\setminus\{0\}, \ \int_{B_1^c} \psi z_\lambda =0\right\}, \end{equation}
Observe that Proposition \ref{key1} implies that the set above is bounded from above, whereas Proposition \ref{G-deg} implies $\Lambda_0 > 0$.
For $\lambda > \Lambda_0$, the Dirichlet problem \eqref{dirichlet2} is nondegenerate, in the sense that the operator $L_\lambda$ has trivial kernel. The following result is rather standard but, since the domain under consideration is unbounded, we prefer to state it and include its proof.
\begin{Lemma}\label{bo} Assume that the operator $L_\lambda$ has trivial kernel. Then:

\begin{enumerate}
\item[a)] The operator $L_\lambda$ is an isomorphism.
\item[b)] Given $v \in H^{1/2}_G(\Sp)$, there exists a unique solution $\psi_v \in H^1_{G}(B_1^c)$ of the problem:
\begin{equation} \label{c00}
\left\{\begin{array} {ll}
- \lambda \Delta \psi_v + \psi_v  - p u_\lambda^{p-1} \psi_v = 0 & \mbox{in }\; B_1^c, \\
               \ \psi_v= v & \mbox{on }\; \pp B_1. \\
\end{array}\right.
\end{equation}
\end{enumerate}
\end{Lemma}

\begin{proof}

We observe that the operator
\[
\psi  \to \lambda \Delta \psi + \psi
\]
is an isomorphism from $H^1_{0,G}(B_1^c) \to H^{-1}_{G}(B_1^c)$. Moreover, the operator
\[
\psi  \to p u_\lambda^{p-1}\psi
\]
is a compact operator from $H^1_{0,G}(B_1^c) \to H^{-1}_{G}(B_1^c)$, because $u_\lambda(x)$ tends to 0 when $|x| \to +\infty$ (see Proposition \ref{list}). Our operator $L_\lambda$ is the sum of the two previous operators, and since it has trivial kernel by assumption, we conclude that it is an isomorphism.

In order to prove b), take $\phi \in H_G^1(B_1^c)$ such that $\phi|_{\partial B_1} = v$. Observe that $\xi= - \lambda \Delta \phi + \phi  - p u_\lambda^{p-1} \phi$ is a element of $H^{-1}_{G}(B_1^c)$ in the sense that:
$$(\xi, \psi)= \int_{B_1^c} \lambda \nabla \phi \cdot \nabla \psi + \phi \, \psi   - p u_\lambda^{p-1} \phi \psi.$$
for all $\psi \in H^1_{0,G}(B_1^c)$. By a), we can find $\theta \in H^1_{0,G}(B_1^c)$ with $L_\lambda(\theta)= \xi$. Then $\phi - \theta$ is a solution of \eqref{c00}.
\end{proof}

%
%

\section{The Dirichlet-to-Neumann operator and its linearization}

The main result of this section is the following:
\begin{Proposition}
\label{dirichle}
Assume that $\lambda > \Lambda_0$, where $\Lambda_0$ is given by \eqref{Lambda0}. Then, for all function $v \in C^{2,\alpha}_{G,m}(\Sp)$ whose norm is sufficiently small, there exists a unique positive solution $u = u (\lambda, v) \in {C}^{2, \alpha} (B_{1+v}^c) \cap H_{0,G}^1(B_{1+v}^c)$ to the problem
\begin{equation}\label{formula}
\left\{
\begin{array}{rcccl}
	- \lambda \Delta u + u-u^p & = & 0 & \textnormal{in} & B_{1+v}^c \\[2mm]
	u & = & 0 & \textnormal{on} & \partial B_{1+v} \,.
\end{array}
\right.
\end{equation}
In addition $u$ depends smoothly on the function $v$, and $u = u_\lambda$ when $v\equiv 0$.
\end{Proposition}

\begin{proof} Let $v  \in C^{2,\alpha}_{G,m}(\Sp)$. Instead of working on a domain depending on the function $v$, it will be more convenient to work on the fixed domain $B_1^c$, endowed with a new metric depending on the function $v$. This can be achieved by considering the diffeomorphism $Y: B_1^c \to B_{1+v}^c$ given by
\begin{equation} \label{Y}
Y ( y) : = \left(1 + \chi ( y) \, v \left(\frac{y}{|y|}   \right) \, \right)y
\end{equation}
where $\chi$ is a cut-off function such that:
\[
\chi(y)=\left \{ \begin{array}{ll} 0 & |y| \geq 3/2, \\ 1 & |y| \leq 5/4. \end{array} \right.
\]
Hence the coordinates we consider from now on are $y \in B_1^c$ and in these coordinates the new metric $g$ can be written as
\[
g  =   \delta_{ij} +  \sum_{i,j} C^{ij} \, dy_i \, dy_j  \, ,
\]
where the coefficients $C^{ij} \in C^{1, \alpha}_G(B_1^c)$ are functions of $y$ depending on $v$ and the first partial derivatives of $v$. Moreover, $C^{ij} \equiv 0$ when $v =0$ and the maps $v \longmapsto C^{ij} (v)$
are smooth. Up to some multiplicative constant, the problem we want to solve can now be rewritten in the form
\begin{equation}
\label{formula-new}
\left\{
\begin{array}{rcccl}
	- \lambda \Delta_g \hat u + \hat u - \hat u^p & = & 0 & \textnormal{in} & B_{1}^c  \\[2mm]
	\hat u & = & 0 & \textnormal{on} & \partial B_{1}  \,.
\end{array}
\right.
\end{equation}
When $v \equiv 0$, the metric $g$ is nothing but the Euclidean metric and a solution of (\ref{formula-new}) is therefore given by $\hat u = u_\lambda$.   In the general case, the relation between the function $u$ and the function $\hat u$ is simply given by
\[
Y^* u =  \hat u\,.
\]
For $\psi \in H_{0,G}^1(B_1^c)$ we define
\[
N (v , \psi) : = - \lambda \Delta_{g} (u_\lambda +\psi) + (u_\lambda +\psi) - [(u_\lambda +\psi)^+]^p\,.
\]
where $(u_\lambda +\psi)^+$ denotes the positive part of the function $u_\lambda +\psi$.
We have
\[
N (0,0) =0.
\]
The mapping $N$ is a smooth map from a neighborhood of $(0,0)$ in $C_{G}^{2, \alpha} (\Sp) \times H_{0,G}^1(B_1^c)$ into $H^{-1}_G(B_1^c)$.
The partial differential of $N$ with respect to $\psi$, computed at $(0,0)$, is given by
\[
\left.D_\psi N\right|_{(0,0)} (\psi) = -\lambda \Delta \psi + \psi -p u_\lambda^{p-1} \psi \,.
\]
Since $\lambda > \Lambda_0$, $\left.D_\psi N\right|_{(0,0)}: H^1_{0,G}(B_1^c) \to H^{-1}_G(B_1^c)$ is an isomorphism by Lemma \ref{bo}. Therefore, the Implicit Function Theorem implies that, for $v$ in a neighborhood of $0$ in $C_{G}^{2, \alpha}(\Sp)$, there exists of $\psi(\lambda, v) \in H^1_{0,G}(B_1^c)$ such that $N(v, \psi(v,\lambda))=0$. The regularity of $u= u_\lambda + \psi(v,\lambda)$ follows from classical Schauder regularity theory, whereas the fact that $u$ is positive comes from the maximum principle.
\end{proof}

For any $\lambda > \Lambda_0$, after canonical identification of $\partial B_{1+v}$ with $\Sp$, we can take an open set $\mathcal{U} \in (\Lambda_0, +\infty) \times C_{G,m}^{2,\alpha}(\Sp)$ containing $(\Lambda_0, +\infty) \times \{ 0\}$, as the domain of the operator $F: \mathcal{U} \to C_{G,m}^{1,\alpha}(\Sp)$ defined by
\begin{equation}\label{F}
F (\lambda,v) =  \displaystyle  \frac{\partial u(\lambda,v)}{\partial \nu}   - \frac{1}{{\rm Vol}(\partial B_{1+v})} \, \int_{\partial B_{1+v}} \, \frac{\partial u(\lambda,v)}{\partial \nu}
\end{equation}
where $\nu$ denotes the unit normal vector field to $\partial B_{1+v}$ pointing to the interior of $B_{1+v}$ and {$u(\lambda,v)$} is the solution of (\ref{formula}) provided by Proposition \ref{dirichle}.

\medskip Observe that $F(\lambda, v)=0$ if and only if $\frac{\partial u}{\partial \nu}$ is constant at the boundary $\partial B_{1+v}$. Clearly, $F(\lambda, 0)=0$ for all $\lambda \in (\Lambda_0, +\infty)$. Our purpose is to find a bifurcation branch from those solutions, so that we get $F(\lambda,\ v_\lambda)=0$, with $v_\lambda \in C_{G,m}^{2,\alpha}(\Sp)$ small, but different from $0$.

 \medskip

We will compute now the Frechet derivative of the operator $F$. 
Before this, we state a useful lemma.

\begin{Lemma} \label{ortog} Let $v\in C^{2,\alpha}_{G,m}(\Sp)$ and $\psi = \psi_v \in C^{2,\alpha}_{G}(B_1^c) \cap H^1_G(B_1^c)$ be a solution of \eqref{c00}. Then:
\[
\int_{B_1^c} \psi z_\lambda=0, \ \ \int_{\partial B_1} \frac{\partial \psi}{\partial \nu}=0.
\]
\end{Lemma}

\begin{proof} We multiply the equation in \eqref{z} by $\psi$, the equation in \eqref{c00} by $z_\lambda$, and integrate by parts to obtain:
$$ \tau_0 \int_{B_1^c} z_\lambda \psi = \int_{\partial B_1}\left( \frac{\partial \psi}{\partial \nu} z_\lambda - \frac{\partial z}{\partial \nu} \psi\right).$$
We know that $z_\lambda=0$ and $\frac{\partial z_\lambda}{\partial \nu}$ is constant on $\partial B_1$ (recall that $z_\lambda$ is radially symmetric) $\psi = v$ on $\partial B_1$. The first identity follows immediately.

Now, define $\kappa_\lambda \in H_r^1(B_1^c)$ as the unique solution of the problem:
\begin{equation} \label{kappa}
\left\{\begin{array} {ll}
- \Delta \kappa_\lambda + \kappa_\lambda  - p u_\lambda^{p-1} \kappa_\lambda = 0 & \mbox{in }\; B_1^c, \\
              \ \kappa_\lambda= 1 & \mbox{on }\; \pp B_1. \\
\end{array}\right. \end{equation}
The existence of such solution is guaranteed for any $\lambda \in \R$ by Proposition \eqref{list} c) and Lemma \ref{bo} b). We multiply \eqref{kappa} by $\psi$, \eqref{c00} by $\kappa_\lambda$ and integrate by parts to conclude:
$$ 0= \int_{\partial B_1} \left( \frac{\partial \psi}{\partial \nu} \kappa_\lambda - \frac{\partial \kappa_\lambda}{\partial \nu} \psi \right).$$
We know that $k_\lambda=1$ and $\frac{\partial k_\lambda}{\partial \nu}$ is constant on $\partial B_1$ ($k_\lambda$ is also radially symmetric), and that  $\psi = v$ on $\partial B_1$. The second identity follows immediately.
\end{proof}

We define the operator $H_\lambda: C^{2,\alpha}_{G,m} (\Sp) \to C^{1,\alpha}_{G,m} (\Sp)$,
\begin{equation}\label{Hache}
 H_\lambda(v) =  \frac{\partial \psi_v}{ \partial \nu}  -(N-1)\,v.
\end{equation}
Here $\psi_v$ is given by Lemma \ref{bo}, b). Observe that by Schauder Elliptic Estimates, if $v \in C^{2, \alpha}_{G,m} (\partial B_1)$, $\psi_v \in C^{2,\alpha}_G (B_1^c)$, and then the operator is well defined.

\begin{Proposition} \label{frechet}
The map $F$ defined in \ref{F} is a $C^1$ operator in a neighborhood of $(\lambda, 0)$ for all $\lambda > \Lambda_0$, and $\left.D_v F\right|_{(\lambda,0)}= H_\lambda$.
\end{Proposition}

\begin{proof}
The operator $F$ is a $C^1$ operator by Proposition \ref{dirichle} (the function $u$ depends smoothly on $v$). The linear operator obtained by linearizing $F$ with respect to $v$ at $(\lambda,0)$ is then given by the directional derivative
\[
F'(w) = \lim_{s \rightarrow 0} \frac{F(\lambda, s\,w) - F(\lambda,0)}{s} = \lim_{s \rightarrow 0} \frac{F(\lambda, s\,w)}{s} .
\]
Writing $v = s \, w$, we consider the diffeomorphism $Y: B_1^c \to B_{1+v}^c$ given in \eqref{Y}.
We set $\hat g$ the induced metric, so that $\hat u = Y^* u$ is the solution (smoothly depending on the real parameter $s$) of
\[
\left\{
\begin{array}{rcccl}
	-\lambda \Delta_{\hat{g}} \, \hat u + \hat u - \hat u^p & = & 0 & \textnormal{in} & B_1^c \\[3mm]
	\hat u & = & 0 & \textnormal{on} & \partial B_1 \,.
\end{array}
\right.
\]
We remark that $\hat u_\lambda  : = Y^* u_\lambda$ is a solution of
\[
-\lambda \Delta_{\hat{g}} \, \hat u_\lambda + \hat u_\lambda - \hat u_\lambda^p  =0
\]
in a neighborhood of $B_1^c$ (note that $u_\lambda$ is radial and then can be extended in a neighborhood of $\pp B_1$), and
\[
\hat u_\lambda (y) = u_\lambda (   (1+ s \, w ) \, y ) \, ,
\]
on $\partial B_1$. Writing $\hat u = \hat u_\lambda + \hat \psi$ we find that
\begin{equation}
\label{f001}
\left\{
\begin{array}{rcccl}
	-\lambda \Delta_{\hat{g}} \, \hat \psi + (\hat u_\lambda + \hat \psi) - (\hat u_\lambda + \hat \psi)^p - \hat u_\lambda + \hat u_\lambda^p & = & 0 & \textnormal{in} & B_1^c \\[3mm]
	\hat \psi & = & - \hat u_\lambda & \textnormal{on} & \partial B_1
\end{array}
\right.
\end{equation}
Obviously $\hat \psi$ is a smooth functions of $s$. When $s =0$, we have $u = u_\lambda$. Therefore,  $\hat \psi=0$ and $\hat u_\lambda = u_\lambda$ when $s=0$. We set
\[
\dot \psi = \partial_s \hat \psi |_{s=0} \,.
\]
Differentiating (\ref{f001}) with respect to $s$ and evaluating the result at $s=0$, we obtain
\[
\left\{
\begin{array}{rlllll}
	-\lambda \Delta\, \dot \psi + (1+pu_\lambda^{p-1})\, \dot \psi & = & 0 & \textnormal{in} & B_1^c \\[3mm]
	\dot \psi & = & - \partial_r u_\lambda \,w & \textnormal{on} & \partial B_1
\end{array}
\right.
\]
where we have set $r := |y|$. To summarize, we have proved that
\[
\hat u = \hat u_\lambda + s \, \partial_r u_\lambda \psi + {\mathcal O} (s^2)
\]
where $\psi$ is the solution of (\ref{c00}). In particular, in $B_{5/4}^c$, we have
\[
\begin{array}{rlll}
\hat u (y) & = & u_\lambda \left( \big(1+ s \, w(y/|y|)\big) \, y \right) + s \, \partial_r u_\lambda\, \psi (y) + {\mathcal O} (s^2) \\[3mm]
                      & =  &u_\lambda (y) + s\,\partial_r u_\lambda\, \big(rw(y/|y|) + \psi \big) + {\mathcal O} (s^2)
\end{array}
\]
In order to compute the normal derivative of the function $\hat u$ when the normal is computed with respect to the metric $\hat g$, we use polar coordinates $y = r \, \theta$ where $\theta \in \Sp$. Then the metric $\hat g$ can be expanded in $B_{5/4}\setminus B_{1}$ as
\[
\hat g = (1 +s w )^2 \, dr^2  + 2 \, s \, (1+s w) \, r \, d w \, dr +  r^2 \,  (1+ s w)^2  \, \mathring e + s^2 \, r^2 \, d w^2
\]
where $\mathring e$ is the metric on $\Sp$ induced by the Euclidean metric. It follows from this expression that the unit normal vector field to $\partial B_1$ for the metric $\hat g$ is given by
\[
\hat \nu  = \left( ( 1 + s \, w )^{-1} +  {\mathcal O} (s^2)\right)  \,  \partial_r + {\mathcal O} (s) \, \partial_{\theta_j}
\]
where $\partial_{\theta_j}$ are vector fields induced by a parameterization of $\Sp$. Using this, we conclude that
\[
\hat g ( \nabla  \hat u , \hat \nu ) =  \partial_r u_\lambda + s \, \left( w \, \,  \partial_r^2 u_\lambda  + \partial_r \psi \right) + {\mathcal O}(s^2)
\]
on $\partial B_1$. The result then follows at once from the fact that $\partial_r u_\lambda$ and $\partial_r^2 u_\lambda$ are constant on $\partial B_1$, while the terms $w$ and $\partial_r \psi $ have mean $0$ on $\partial B_1$, and
\[
-\lambda (\partial_r^2 u_\lambda + (N-1) \partial_r u_\lambda) = 0
\]
on $\partial B_1$.
This completes the proof of the proposition.
 \end{proof}

\section{Study of the linearized operator}

In view of Proposition \ref{frechet}, a bifurcation might appear only for values of $\lambda$ so that $H_\lambda$ becomes degenerate ($H_\lambda$ was defined in \eqref{Hache}). We shall see that this is indeed the case for some $\lambda > \Lambda_0$. Let us define the first eigenvalue of the operator $H_\lambda$as
\[
\sigma_1 (H_\lambda) = \inf \left\{ \int_{\Sp} v\,H_\lambda(v) \, \, \, :\, \, \, v \in C^{2,\alpha}_{G,m} (\Sp) \,,\,\,\, \int_{\Sp} v^2 =1 \right\} \in [-\infty, +\infty).
\]
The main result of this section is the following:
\begin{Proposition}\label{linearize} There exists $\Lambda_2> \Lambda^* > \Lambda_0$ ($\Lambda_0$ is given in \eqref{Lambda0}) such that:
\begin{enumerate}
\item if $\lambda \geq \Lambda_2$ then $\sigma_1(H_\lambda) >0$;
\item $\sigma_1(H_{\Lambda^*}) =0$;
\item there exists a sequence of real numbers $\lambda_n$, increasing and converging to $\Lambda^*$, such that $\sigma_1(H_{\lambda_n}) < 0$.
\end{enumerate}
\end{Proposition}
In order to prove Proposition \ref{linearize}, let us define the following bilinear forms: $T_\lambda : H^{1/2}_{G}(\Sp) \times H^{1/2}_{G}(\Sp) \to \R$ defined by
\[
T_\lambda(v_1, v_2) = \int _{\Sp} v_1 \frac{\partial \psi_{v_2}}{\partial \nu} \,,
\]
and $\tilde T_\lambda : H^{1/2}_{G}(\Sp) \times H^{1/2}_{G}(\Sp) \to \R$ defined by
\[
\tilde T_\lambda(v_1, v_2) =  \int _{\Sp} v_1 \frac{\partial \psi_{v_2}}{\partial \nu} - \int _{\Sp} (N-1) v_1 v_2\,.
\]
Observe that $T_\lambda$, $\tilde T_\lambda$ are symmetryc. Moreover, it is clear that:
\[
\sigma_1 (H_\lambda) = \inf \left\{ \int_{\Sp} \tilde{T}_\lambda(v,v) \, \, \, :\, \, \, \ v \in C^{2,\alpha}_{G,m}(\Sp),\ \int_{\Sp} v =0 ,\ \int_{\Sp} |v|^2= 1 \right\}.
\]
We also define the bilinear form $\tilde Q_\lambda: H^1_G(B_1^c) \times H^1_G(B_1^c) \to \R$, by
\begin{equation} \label{bilinear}  \tilde Q_\lambda(\psi_1, \psi_2)= Q_\lambda(\psi_1, \psi_2) - \lambda\,(N-1)\int_{\partial B_1} \psi_1 \psi_2 \, , \end{equation}
where $Q_\lambda$ has been defined in \eqref{bilinear0}. It is easy to verify that
\[
T_\lambda(v_1, v_2) = \frac{1}{\lambda}\, Q_\lambda(\psi_{v_1}, \psi_{v_2})\ , \ \ \tilde T_\lambda(v_1, v_2) = \frac{1}{\lambda}\, \tilde Q_\lambda(\psi_{v_1}, \psi_{v_2}).
\]
The following lemma relates $\sigma_1(H_\lambda)$ with the bilinear form $\tilde{Q}$.

\begin{Lemma} For any $ \lambda > \Lambda_0$ we have
\[
\sigma_1(H_\lambda) =  \inf \left \{  \frac{1}{\lambda}\, \tilde Q(\psi, \psi)\,\,\, :\,\,\, \psi \in E \, , \, \int_{\partial B_1} \psi^2= 1\right \}.
\]
where
\begin{equation} \label{E} E= \left \{ \psi \in H^1_G(B_1^c), \ \int_{\partial B_1}\psi =0, \ \int_{B_1^c} \psi z_{\lambda} =0\right\}. \end{equation}
Moreover this infimum is attained.
\end{Lemma}

\begin{proof}

Fix $ \lambda > \Lambda_0$. First we prove that
\begin{equation}\label{gamma1}
\gamma_1:=\inf \left \{Q_\lambda(\psi, \psi)\,\,\, :\,\,\, \psi \in E \, , \, \int_{\partial B_1} \psi^2= 1\right \}.
\end{equation}
is achieved. Take $\psi_n \in E$ such that $Q_\lambda(\psi_n, \psi_n) \to \gamma_1 \in [-\infty, +\infty)$. We show that $ \psi_n $ is bounded by contradiction; if $\| \psi_n \| \to +\infty$, define $\phi_n = \| \psi_n\|^{-1} \psi_n$, and we can assume that up to a subsequence $\phi_n \weakto \phi_0$. Observe that $\int_{\partial B_1} \phi_n^2 \to 0$, which implies that $\phi_0 \in H_{0,G}^1(B_1^c)$. We also point out that $$ \int_{B_1^c} u_\lambda^{p-1} \phi_n^2 \to \int_{B_1^c} u_\lambda^{p-1} \phi_0^2.$$ Now, let us consider two cases:\\
{\bf Case 1:} $\phi_0=0$. In such case,
$$ Q_\lambda(\psi_n, \psi_n)= \| \psi_n\|^2 \int_{B_1^c} \left( \lambda |\nabla \phi_n|^2 + \phi_n^2 - p u_\lambda^{p-1} \phi_n^2 \right)  \to + \infty,$$
which is impossible.\\
{ \bf Case 2:} $\phi_0 \neq 0$. In this case,
\begin{eqnarray*}
\liminf_{n \to + \infty} Q_\lambda(\psi_n, \psi_n) & =  & \liminf_{n \to + \infty} \| \psi_n\|^2  \int_{B_1^c} \left(\lambda |\nabla \phi_n|^2 + \phi_n^2 - p u_\lambda^{p-1} \phi_n^2 \right) \\
&  \geq & \liminf_{n \to + \infty} \| \psi_n\|^2 Q_\lambda(\phi_0, \phi_0),
\end{eqnarray*}
but $Q_\lambda(\phi_0, \phi_0)>0$ since $\lambda > \Lambda_0$. This is again a contradiction.\\
Therefore, $\psi_n$ is bounded, so up to a subsequence we can pass to the weak limit $\psi_n \weakto \psi$. As before,
$$  1=\int_{\partial B_1} \psi_n^2 \to \int_{\partial B_1} \psi^2, \  \int_{B_1^c} u_\lambda^{p-1} \psi_n^2 \to \int_{B_1^c} u_\lambda^{p-1} \psi^2.$$
Then $\psi$ is a minimizer for $\gamma_1$, and in particular $\gamma_1 > -\infty$.
\medskip

Now we observe that under the constraints $\psi \in E$, and $\int_{\partial B_1} \psi^2 =1$ we have
\begin{equation}\label{rel1}
\tilde Q_\lambda(\psi,\psi) = Q_\lambda(\psi,\psi) - \lambda\, (N-1)
\end{equation}
and in particular, also
\[
\inf \left \{\frac{1}{\lambda}\,\tilde Q_\lambda(\psi, \psi)\,\,\, :\,\,\, \psi \in E \, , \, \int_{\partial B_1} \psi^2= 1\right \}.
\]
is achieved.

\medskip

Let $\psi \in E$ be the minimizer such that $Q_\lambda(\psi,\psi)= \gamma_1$. By the Lagrange multiplier rule, there exist $ \alpha_0$, $\alpha_1$,  $\alpha_2 \in \R$ so that for any $\rho \in H^1_{G}(B_1^c)$,
$$ \int_{B_1^c} \left( \nabla \psi \cdot \nabla \rho + \psi \rho - p u_\lambda^{p-1} \psi \rho - \alpha_0 \rho z_\lambda\right)= \int_{\partial B_1} \rho (\alpha_1 \psi + \alpha_2).$$
Taking $\rho = \psi$, we conclude that $\alpha_1 = \gamma_1$. Moreover, taking $\rho = z_\lambda$ and $\rho = \kappa_\lambda$ (recall the definitions of $z_\lambda$ and $k_\lambda$ in \eqref{z} and \eqref{kappa}), we conclude that $\alpha_0=0$ and $\alpha_2=0$, respectively. In other words, $\psi$ is a (weak) solution of the equation:
\begin{equation} \label{steklov}
\left\{\begin{array} {ll}
- \lambda \Delta \psi + \psi  - p u_\lambda^{p-1} \psi = 0 & \mbox{in }\; B_1^c, \\
               \ \frac{\partial \psi}{\partial \nu}= \gamma_1 \psi & \mbox{on }\; \pp B_1. \\
\end{array}\right.
\end{equation}
By the regularity theory, $\psi \in C^{2,\alpha}_G(B_1^c)$.

\medskip Now recall that $T_\lambda(v, v) = \frac{1}{\lambda}Q_\lambda(\psi_{v}, \psi_{v})$. By Lemma \ref{ortog} $\psi_v \in C^{2,\alpha}_G(B_1^c) \cap  E$, and then
$$\gamma_1 \leq  \inf \left \{\lambda\,  T_\lambda(v, v):\ v \in C^{2,\alpha}_G(\Sp),\ \int_{\Sp} v =0 ,\ \int_{\Sp} |v|^2= 1\right \}$$
Moreover, $\gamma_1$ is achieved at a certain $\psi \in C^{2,\alpha}_G(B_1^c)$, which solves \eqref{steklov}. In particular, denoting $v= \psi|_{\partial B_1}$, we conclude that $\lambda\, T_\lambda(v,v)= \gamma_1$. Then we have
\begin{equation} \label{gamma1bis} \gamma_1 =  \inf \left \{ \lambda\, T_\lambda(v, v):\ v \in C^{2,\alpha}_G(\Sp),\ \int_{\Sp} v =0, \ \int_{\Sp} |v|^2= 1\right \}. \end{equation}
Now we observe that under the constraints $\int_{\Sp} v =0$, and $\int_{\Sp} |v|^2= 1$ we have
\[
\tilde T_\lambda(v,v) = T_\lambda(v,v) - (N-1)
\]
and then the result follows at once from \eqref{gamma1}, \eqref{rel1} and \eqref{gamma1bis}.
\end{proof}

The previous lemma leads us to the study of the bilinear form $\tilde Q_\lambda$. The first key result for our purposes is the following:

\begin{Proposition} \label{key} There exists $M>\Lambda_0$ such that for any $\lambda > M$,  $\tilde Q_\lambda(\psi, \psi)>0 $ for any $\psi \in E\setminus\{0\}$, where $E$ is the subspace defined in \eqref{E}.
 \end{Proposition}
The proof of this proposition is somehow delicate and it is postponed to Section \ref{ultima}. We point out that this is the only point where the assumption $p< \frac{N+2}{N-2}$ (if $N>2$) is needed.

\medskip

Let us define now:
\begin{equation} \label{Lambda*} \Lambda^* = \sup \{ \lambda >0: \tilde Q_\lambda(\psi, \psi) < 0 \ \mbox{ for some } \psi \in E\}. \end{equation}
By Proposition \ref{key}, $\Lambda^*<+\infty$. Moreover, since $\tilde Q_\lambda(\psi,\psi) = Q_\lambda(\psi, \psi)$ for all $\psi \in H^1_{0,G}(B_1^c)$, we have also that $\Lambda^* \geq \Lambda_0$. The last main ingredient to prove Proposition \ref{linearize} is the following:

\begin{Lemma}\label{maj} $\Lambda^* > \Lambda_0$ \end{Lemma}

\begin{proof} It suffices to show that for $\lambda= \Lambda_0$, $\tilde Q_\lambda(\psi, \psi)<0$ for some $\psi \in E$. Reasoning by contradiction, assume that $\tilde Q_\lambda$ is semipositive definite in $E$.
By definition of $\Lambda_0$, there exists $\psi_0 \in H_{0,G}^1(B_1^c)$, with $Q_\lambda(\psi_0,\psi_0)=0$, and $\psi_0$ is a solution of \eqref{linear-d}. We have $\psi_0 \in E$ and by our assumptions it is also a minimizer for $\tilde Q_\lambda$ when defined in $E$. By the Lagrange multiplier rule, there exist $ \alpha_0$ and $\alpha_1 \in \R$ so that for any $\rho \in H^1_{G}(B_1^c)$,
$$ \int_{B_1^c} \left( \nabla \psi \cdot \nabla \rho + \psi \rho - p u_\lambda^{p-1} \psi \rho - \alpha_0 \rho z_\lambda\right)= \alpha_1 \int_{\partial B_1} \rho $$
Taking $\rho = z_\lambda$ and $\rho = \kappa_\lambda$ (recall the definitions of $z_\lambda$ and $k_\lambda$ in \eqref{z} and \eqref{kappa}), we conclude that $\alpha_0=0$ and $\alpha_1=0$, respectively. In other words, $\psi$ is a (weak) solution of the equation:
$$
\left\{\begin{array} {ll}
-\lambda \Delta \psi_0 + \psi_0 - p u_\lambda^{p-1} \psi_0 = 0  & \mbox{in }\; B_1^c, \\
               \ \frac{\partial \psi_0}{\partial \eta} - (N-1) \psi_0= 0 & \mbox{on }\; \pp B_1. \\
\end{array}\right.
$$
Since $\psi_0=0$ on $\partial B_1$, we have $\frac{\partial \psi_0}{\partial \nu}=0$ on $\partial B_1$. By unique continuation we should have $\psi_0=0$, but this is a contradiction.
\end{proof}

We are now able to give the proof of the main proposition of this section.

\begin{proof}
{\it (Proposition \ref{linearize}.)} Assertion (1) follows immediately from Proposition \ref{key}. Statements (2) and (3) follow by the definition of $\Lambda^*$ in \eqref{Lambda*} and Lemma \ref{maj}.
\end{proof}

\section{The bifurcation argument}

In order to use a local bifurcation result we need to rewrite our problem in a more convenient way. For that, the following lemma will be essential.

\begin{Lemma} \label{inversion} There exists $\e>0$ such that for any $\lambda \in (\Lambda^*-\e, +\infty)$, the operator
\[
\begin{array}{ccccc}
H_\lambda + Id &: & C^{2,\alpha}_{G,m} (\Sp) & \to &  C^{1,\alpha}_{G,m} (\Sp)\\[2mm]
& & v & \mapsto & H_\lambda(v) + v
\end{array}
\]
is invertible.
\end{Lemma}

\begin{proof} It suffices to prove that the operator
\[
v \to  H_\lambda(v) + \sigma\, v
\]
defined in $C^{2,\alpha}_{G,m}(\Sp)$ is invertible for all $\sigma > -\sigma_1(H_\lambda)$. Equivalently, we can prove that the operator
\[
v \to \left.\frac{\partial \psi_v}{ \partial \nu}\right|_{\partial B_1}  +\gamma\,v
 \]
defined in $C^{2,\alpha}_{G,m}(\Sp)$ is invertible for all $\gamma > -\gamma_1$, where $\gamma_1$ is defined in \eqref{gamma1}.
Then, define the bilinear form $Q_{\lambda, \gamma}: E \times E \to \R$ as
\[
Q_{\lambda,\gamma}(\psi_1, \psi_2)= Q_\lambda(\psi_1, \psi_2) + \lambda\, \gamma \int_{\partial B_1}\psi_1\, \psi_2,
\]
and the bilinear form $T_{\lambda,\gamma} : H^{1/2}_G(\Sp) \times  H^{1/2}_G(\Sp) \to \R$ as
\[
T_{\lambda,\gamma} (v_1,v_2) = T_\lambda(v_1, v_2) + \gamma \int_{\Sp} v_1\, v_2.
\]
Since $\gamma > \gamma_1$, those bilinear forms are positive definite. We claim that they are indeed coercive. Let us start with $Q_{\lambda,\gamma}$, and show that:
$$ \alpha:= \inf\{ Q_{\lambda,\gamma}(\psi, \psi):\ \psi \in E, \|\psi\|=1\} >0.$$
Take $\psi_n \in E$, $\|\psi_n\|=1$, $Q_{\lambda,\gamma}(\psi_n,\psi_n) \to \alpha$, and assume that $\psi_n \weakto \psi_0$. If the convergence is strong, then the infimum $\alpha$ is attained, which implies that $\alpha >0$. Otherwise,
\begin{align*} \alpha = \limsup_{n\to +\infty} \int_{B_1^c} \lambda |\nabla \psi_n|^2 + \psi_n^2  - p  u_\lambda^{p-1} \psi_n^2  + \gamma \int_{\partial B_1} \psi_n^2 \\ > \int_{B_1^c} \lambda |\nabla \psi_0|^2 + \psi_0^2 - p  u_\lambda^{p-1} \psi_0^2  + \gamma \int_{\partial B_1} \psi_0^2 \geq 0.\end{align*}
Hence $Q_{\lambda,\gamma}$ is coercive. Now, observe that:
\begin{align*} T_{\lambda,\gamma} (v,v) = \int _{\partial B_1} \left[v \frac{\partial \psi_{v}}{\partial \nu}  +\gamma \,v^2\right]  = \frac{1}{\lambda}Q_{\lambda,\gamma} (\psi_{v}, \psi_{v}) \geq c \|\psi_v\|^2_{H^1(B_1^c)} \geq c' \|v\|^2_{H^{1/2}(\Sp)} , \end{align*}
where we have used the trace estimate in the last inequality. Therefore $T_{\lambda,\gamma}$ is coercive. By the Lax-Milgram Theorem, the operator
\[
v \to \left.\frac{\partial \psi_v}{ \partial \nu}\right|_{\partial B_1}  +\gamma\,v
 \]
 is invertible for all $\gamma > -\gamma_1$ in the spaces $H^{1/2}_G(\Sp) \to H^{-1/2}_G(\Sp)$. By the regularity theory and the fact that the mean property is preserved, it is invertible also in the spaces $C^{2,\alpha}_{G,m}(\Sp) \to C^{1,\alpha}_{G,m}(\Sp)$.
\end{proof}

According to Proposition \ref{linearize}, we can take $\Lambda_1 \in (\Lambda_0, \Lambda^*)$ sufficiently close to $\Lambda^*$ so that $\sigma_1(H_{\Lambda_1})<0$.  We define $G: [\Lambda_1, \Lambda_2] \times \mathcal{V} \to \mathcal{W}$ by
\begin{equation}\label{GG}
G (\lambda,v) =  F(\lambda, v) + v.
\end{equation}
Here $\mathcal{V} \subset C_{G,m}^{2,\alpha}(\Sp)$ and $\mathcal{W} \subset C_{G,m}^{1,\alpha}(\Sp)$ are open neighborhoods of the $0$ function, and $\Lambda_2$ is given by Proposition \ref{linearize}.
By Lemma \ref{inversion}, taking $\Lambda_1$ close enough to $\Lambda^*$ we can assume that $\left.D_vG\right|_{(\lambda, 0)}$ is an isomorphism for all $\lambda \in [\Lambda_1, \Lambda_2]$. By using the Inverse Function Theorem, we can further restrict $\mathcal{V}$ and $\mathcal{W}$ so that $G(\lambda, \cdot)$ is invertible for all $\lambda \in [\Lambda_1, \Lambda_2] $.

\medskip

Define now $R: [\Lambda_1, \Lambda_2] \times \mathcal{W} \to \mathcal{W}$, $R(\lambda, w)= w - \tilde w$, where $\tilde w$ is such that $G(\lambda, \tilde w) = w$. We point out that $R$ has the form of identity plus a compact operator. Clearly, $F(\lambda, v)=0 \Leftrightarrow R(\lambda, v)=0$. Hence Theorem \ref{main} follows if we show local bifurcation of solutions of the equation $R(\lambda, v)=0$.

\medskip
We have
$$\left.D_w R\right|_{(\lambda, 0)}(w)= w - \left.D_w G\right|_{(\lambda, 0)}^{-1}(w).$$
Hence
\begin{equation} \label{RtoH} \left.D_w R\right|_{(\lambda, 0)}(w)= \mu w  \Leftrightarrow H_\lambda(w)= \frac{\mu}{1-\mu} w.\end{equation}
By the proof of Lemma \ref{inversion}, $\mu<1$ if $\lambda \geq \Lambda_1$. Therefore $\left.D_w R\right|_{(\lambda, 0)}(w)$ has the same number of negative eigenvalues as $H_\lambda$.

\medskip

In this framework we can use a local bifurcation result by Krasnoselskii.
\begin{Theorem}\label{Kr} (see for instance \cite{kielhofer}, [II.3.2]). Let $F: [a,b] \times \mathcal{Z} \to X$ be $C^1$ map defined in $\mathcal{Z} \subset X$ a neighborhood of the origin in the Banach Space $X$. Assume that $F$ is given by $F(\lambda, x)= x- K(\lambda, x)$ where $K(\lambda, \cdot)$ is a compact map. Assume moreover, that $\left.D_xF\right|_{(a, 0)}$ and $\left.D_xF\right|_{(b, 0)}$ are isomorphisms of $X$. Denote by $i_{D_xF}(a)$ and $i_{D_xF}(b)$ their indices, that is, the number of negative eigenvalues (counted with algebraic multiplicity). Assume finally that $i_{D_xF}(a)-i_{D_xF}(b)$ is an odd integer. Then every neighborhood of $[a, b] \times \{0\}$ contains solutions of $F(\lambda,x) = 0$, with $\lambda \in (a,b)$, $x \in X$, $x \neq 0$.
\end{Theorem}

\begin{Remark}\rm{ The above version of the Krasnoselskii bifurcation result differs slightly from the classical one; usually one imposes the existence of an unique value $\lambda \in (a,b)$ such that the derivative $\left.D_xF\right|_{(\lambda, 0)}$ is degenerate. Under this assumption, one concludes bifurcation at the point $(\lambda,0)$. The version we give above follows immediately from the proof of the classical Krasnoselskii bifurcation result, which is based on a change of the Leray-Schauder degree of the $0$ solution. A drawback of this version is that we cannot identify exactly the bifurcation point.}
\end{Remark}
We now apply Theorem \ref{Kr} to $R(\lambda, w)$. For $\lambda= \Lambda_2$, $i_{D_vR}(\Lambda_2)=0$ by Proposition \ref{linearize}. Therefore we just need to show the validity of the following Lemma.

\begin{Lemma}\label{caracola} $i_{D_vR}(\Lambda_1)$ is odd if $\Lambda_1$ is chosen sufficiently close to $\Lambda^*$.
\end{Lemma}

\begin{proof} In view of \eqref{RtoH}, it suffices to prove that $H_{\Lambda^*}$ has a kernel with odd multiplicity. For any $\psi \in E$, there exist functions $\psi_0, \psi_{k,j}$ defined in $[1,+\infty)$ such that we can write
\[
\psi(r, \theta) = \psi_0(r) + \sum_{k=1}^{+\infty} \sum_{j=1}^{\tilde m_k} \psi_{k,j}(r) \,\zeta_{k,j}(\theta)\,,
\]
where $r = |x|$, $\theta = \frac{x}{|x|}$ and $\zeta_{k,j}$ are the $G$-symmetric spherical harmonics (normalized to 1 in the $L^2$-norm) with eigenvalue $\mu_{i_k}$ of multiplicity $m_k$. Then the quadratic form $\psi \to \tilde Q_\lambda(\psi, \psi)$ defined in $E$ can be written as
\begin{equation}\label{decomp}
\tilde Q_\lambda(\psi, \psi) = \tilde Q_{\lambda,0}(\psi_0) +  \sum_{k=1}^{+\infty} \sum_{j=0}^{\tilde m_k} \tilde Q_{\lambda,k}(\psi_{k,j})
\end{equation}
where for a function $\phi: (1,+\infty) \to \R$ we denote
\begin{eqnarray*}
\tilde Q_{\lambda,k}(\phi) & = & \int_1^{+\infty} \left( \lambda\, \phi'\, \phi' + \phi^2- p\, u_\lambda^{p-1}\, \phi^2 \right)\, r^{N-1}\, dr - (N-1) \, \phi(1)^2\\
& & \qquad +  \mu_{i_k}\int_1^{+\infty} \phi^2 \, r^{N-3}\, dr
\end{eqnarray*}
choosing by convention that $\mu_{i_0}=0$.
Since $\psi \in E$ we have that $\psi_0(1)=0$ and that $\psi_0$ is orthogonal to the function $z_\lambda$ restricted to the radial variable. By Proposition \ref{list} we have $\tilde Q_{\lambda,0}(\psi_0) > 0$. For $\lambda=\Lambda^*$, the bilinear form $\tilde Q_\lambda$ is positive semi-definite in $E\times E$, and then from \eqref{decomp} we have that all the quadratic forms $\tilde Q_{\lambda,k}$ are positive semi-definite. Moreover, it is clear that
\[
\tilde Q_{\lambda, k_1}(\phi)  < \tilde Q_{\lambda,k_2}(\phi)
\]
if $1\leq k_1< k_2$. We know also that there exists a $\psi \in E$ such that $\tilde Q_\lambda(\psi,\psi)=0$. Therefore $\tilde Q_{\lambda,1}$ is positive semi-definite, and $\tilde Q_{\lambda,k}$ are positive definite for $k>1$. This implies that the kernel of $H_{\Lambda^*}$ has dimension equal to $m_1$, which is odd by assumption \ref{G}.
\end{proof}


\section{Proof of Propositions \ref{key1} and \ref{key} \label{ultima}}

Observe that the bilinear form $\tilde{Q}_\lambda$ defined in \eqref{bilinear}, when restricted to functions in $H_{0,G}^1(B_1^c)$, is nothing but $Q_\lambda$ (recall \eqref{bilinear0}). Hence Proposition \ref{key1} follows immediately from Proposition \ref{key}.

In order to prove Proposition \ref{key}, we shall consider the problem in the form \eqref{dirichlet}; that is, we aim to prove that $\hat{Q}_R: H^1_G(B_R^c) \times H^1_G(B_R^c) \to \R$,
$$  \hat{Q}_R(\psi_1, \psi_2)= \int_{B_R^c} \nabla \psi_1 \cdot \nabla \psi_2 + \psi_1 \, \psi_2 - p u_{R}^{p-1} \psi_1 \, \psi_2 - \frac{N-1}{R} \int_{\partial B_R} \psi_1 \psi_2  $$
is positive definite if $R>0$ is sufficiently small, when $\psi_1$, $\psi_2$ belong to the space:
$$ E_R= \left \{ \psi \in H^1_G(B_R^c), \ \int_{\partial B_R}\psi =0, \ \int_{B_R^c} \psi z_{R} =0\right\}. $$
Here $u_R$ and $z_R$ stand for
$$ u_R(x)= u_\lambda\left(\frac{x}{R}\right), \ \ z_R(x)=  z_\lambda\left(\frac{x }{R}\right),  \quad \lambda= R^{-2}.$$
The strategy of the proof is to make $R=R_n \to 0$ to and assume that $\hat{Q}_R$ is not positive definite in $E_R$ to reach a contradiction. For that, the behavior of $u_R$, $z_R$ as $R \to 0$ is needed. This result might be known, but we have not been able to find a specific reference.

\begin{Lemma} Let $u_n$ be the positive radial solution of \eqref{dirichlet} for $R=R_n \downarrow 0$, and $z_n=\| z_{R_n}\|^{-1} z_{R_n}$. Let us consider those functions extended to $\RN$ by $0$. Then, $u_n \to U$ and $z_n \to Z$ in $H^1(\RN)$, where $U$ is the radial ground state solution of problem:
\begin{equation} \label{entera}
-\Delta U +U = U ^p,\ \  U >0, \  \mbox{ in }\; \R^N,
\end{equation}
and $Z$ is the normalized positive eigenfunction corresponding to the negative eigenvalue of the linearized problem, that is,
\begin{equation} \label{entera2}
-\Delta Z +Z - p U^{p-1}Z = \tau Z,\ \\  \mbox{ in }\; \R^N,
\end{equation}
with $\tau<0$.
\end{Lemma}

\begin{proof}

Let us define the energy functional associated to \eqref{dirichlet}:
$$ I(u)= \frac 1 2 \left (\int_{B_{R_n}^c} |\nabla u|^2 + u^2 \right ) - \frac{1}{p+1} \int_{B_{R_n}^c} |u|^{p+1}.$$
It is well known that, $$I(u_n)= \inf \{ \max \{I(tu): t>0 \}\,, u \in H_{0,r}^1(B_{R_n}^c) \} >0,$$ see for instance \cite{sz}. Since $H_{0,r}^1(B_{R_n}^c) \subset H_{0,r}^1(B_{R_{n+1}}^c)$ (up to extension by $0$), then $I(u_n)$ is decreasing in $n$. In particular, $I(u_n)$ is bounded.
Moreover, multiplying \eqref{dirichlet} by $u_n$ and integrating, we obtain that $DI_{u_n}(u_n)=0$. What follows is standard (see for instance \cite{ambr-rab}); first, observe that:
$$(p+1)I(u_n) =(p+1)I(u_n) - DI_{u_n}(u_n) = \frac{p-1}{2} \|u_n\|^2,$$
frow which $u_n$ is bounded. Passing to a subsequence, we can assume that $u_n \weakto  u_0$ in $H^1$ sense. Multiplying \eqref{dirichlet} by $\phi \in C_0^{\infty}(\R^N \setminus\{0\})$, we conclude that $u_0$ is a weak solution of the problem:
\begin{equation} \label{u0} - \Delta u_0 + u_0 = u_0^p \ \mbox{ in } \R^N \setminus \{0\}. \end{equation}
Since $u_0$ is in the Sobolev class, the singularity is removable.
Multiplying \eqref{dirichlet} by $u_n$ and \eqref{u0} by $u_0$, we have:
$$ \| u_n \|^2 = \int_{\R^N} |u_n|^{p+1}, \ \ \| u_0 \|^2 = \int_{\R^N} |u_0|^{p+1}.$$
By \cite{strauss}, $u_n \to u_0$ strongly in $L^{p+1}$, so that $\|u_n\| \to \|u_0\|$. From this we conclude that $u_n \to u_0$ strongly in $H^1(\R^N)$ and $u_0$ is a nontrivial positive solution of \eqref{entera}. By uniqueness (\cite{kwong}), $u_0=U$.
Regarding $z_n$, it is a radially symmetric function solving:
\begin{equation} \label{z2}
\left\{\begin{array} {ll}
-\Delta z_n + z_n  - p u_n^{p-1} z_n = \tau_n z_n & \mbox{in }\; B_{R_n}^c, \\
         \      z_n= 0 & \mbox{on }\; \pp B_{R_n}, \\
\end{array}\right.
\end{equation}
with $\tau_n <0$. Since $\|z_n\|=1$, $z_n$ converges weakly to some $z_0$. Multiplying the above equation by $z_n$ we get:
\begin{equation}\label{hi} \int_{\RN} |\nabla z_n|^2 + (1-\tau_n)z_n^2 -p u_n^{p-1} z_n^2 =0.\end{equation}
By compact embedding of radial functions (\cite{strauss}), for instance, we conclude that:
\begin{equation}\label{hi2} \int_{\RN} u_n^{p-1} z_n^2 \to \int_{\RN} U^{p-1} z_0^2. \end{equation}
This implies in particular that $z_0$ is not zero. Moreover, $\liminf_{n \to +\infty} \int_{\RN} z_n^2 \geq \int_{\RN} z_0^2$. In particular $\tau_n$ is bounded, and we can assume $\tau_n \to \tau_0 \leq 0$. Then, $z_0$ is a weak solution of
$$-\Delta z_0 + z_0  - p U^{p-1} z_0 = \tau_0 z_0 \ \  \mbox{in }\; \RN \setminus\{0\}.$$
Since $z_0$ belongs to the Sobolev class, the singularity is removable, and it is an entire solution; hence $z_0 = Z$. In particular,
$$\int_{\RN} |\nabla z_0|^2 + (1-\tau_0)z_0^2 -p U^{p-1} z_0^2 =0.$$
This, together with \eqref{hi} and \eqref{hi2}, allows us to conclude that $z_n \to z_0$ strongly, concluding the proof.
\end{proof}

The following lemmas will be of use:

\begin{Lemma} \label{technical} For any $f\in C_0^{\infty}(\R)$, the following inequality holds:
$$  r^{N-2} f(r)^2 \leq \frac{1}{\lambda} \int_r^{+\infty} f'(s)^2 s^{N-1} \, ds + (2-N + \lambda)  \int_r^{+\infty} f(s)^2 s^{N-3} \, ds, $$
where $N\geq 2$, $\lambda>0$ and $r>0$.

\end{Lemma}

\begin{proof} Observe that: $$ r^{N-2} f(r)^2 = - \int_r^{+\infty} (2 s^{N-2} f(s) f'(s) + (N-2) s^{N-3} f(s)^2)\, ds.$$
We now estimate the first term in the right hand side by using Cauchy-Schwartz inequality:
\begin{align*} \int_r^{+\infty} 2 s^{N-2} |f(s)| |f'(s)|\, ds = \int_r^{+\infty} 2 s^{\frac{N-3}{2}} |f(s)|  s^{\frac{N-1}{2}} |f'(s)|\, ds \\
\leq \lambda  \int_r^{+\infty} s^{N-3} f(s)^2\, ds + \frac{1}{\lambda} \int_r^{+\infty} s^{N-1} f'(s)^2\, ds\,. \end{align*}
This lemma follows at once.
\end{proof}

\begin{Lemma} \label{appendix} Let $G$ be a group of symmetries satisfying \ref{G}. Then,
$$ \frac{1}{R} \int_{\partial B_R} \psi(x)^2 \, ds_x \leq \frac{1}{N} \int_{B_R^c} |\nabla \psi(x)|^2 \, dx $$
for any $\psi \in H^1_G(B_R^c)$ with $\int_{\partial B_R} \psi(x)\, ds_x=0$.
\end{Lemma}

\begin{proof}

By density arguments, we can assume that $\psi \in C^{\infty}_0(B_R^c)$. We decompose it in Fourier series:
$$\psi(r, \theta)= \sum_{i=0}^\infty \psi_k(r) \phi_k(\theta),$$
where $\phi_k$ are the eigenfunctions of $\Delta_{\Sp}$ under $G-$symmetry. Observe that $\phi_0(\theta)=1$ and $\psi_0(\frac{1}{R})=0$. Therefore it suffices to prove the inequality for the summands $\psi_k(r) \phi_k(\theta)$, $i\geq 1$. Observe that:
$$ \int_{B_R^c} |\nabla ( \psi_k(r) \phi_k(\theta))|^2 \, dr \, d\theta= \int_R^{+\infty} ( \psi_k'(r)^2 r^{N-1} + {\mu}_{i_k} \psi_k(r)^2 r^{N-3})\, dr \ \int_{\partial B_1} \phi_k(\theta)^2 \, d \theta.$$
Moreover,
$$ \frac{1}{R} \int_{\partial B_R} |\psi_k(r) \phi_k(\theta)|^2\, ds_x = R^{N-2} \psi_k(R)^2  \int_{\partial B_1} \phi_k(\theta)^2 \, d \theta.$$
By assumption \ref{G}, ${\mu}_{i_1} \geq \mu_2 = 2 N$. Now it suffices to take $\lambda=N$ in Lemma \ref{technical} to conclude.
\end{proof}

\medskip

We are now able to prove Proposition \ref{key}.

\medskip

\begin{proof} {\it (Proposition \ref{key})}
Take $R=R_n \to 0$, denote $B_n=B_{R_n}$, $u_n= u_{R_n}$ and $z_n= z_{R_n}$, and define:
$$ \chi_n = \inf \left \{ \hat{Q}_R(\psi, \psi):\ \psi \in H^1_G(B_n^c),\ \int_{\partial B_n}\psi =0, \ \int_{B_n^c} \psi z_n =0,\ \int_{B_n^c} |\psi|^2= 1\right \}.$$
Assume, by contradiction, that $\chi_n \leq 0$. The proof will be divided in several steps:

\medskip {\bf Step 1:} We show here that $\chi_n $ is attained.\\
Take $\psi_k$ a minimizing sequence for $\chi_n$. If $\psi_k$ is unbounded in the $H^1$ norm, define $\phi_k = \|\psi_k\|^{-1}\psi_k$. Then,
$$0 \leq \int_{B_n^c} | \nabla \phi_k|^2 + (1-\chi_n) |\phi_k|^2 - p u_n^{p-1} |\phi_k|^2 - \frac{N-1}{R_n} \int_{\partial B_n} |\phi_k|^2 \to 0. $$
But $\phi_k \to 0$ in the $L^2$ norm, so that $\int_{B_n^c} u_n^{p-1} |\phi_k|^2 \to 0$ as $k \to + \infty$. Moreover, $\phi_k \weakto 0$ in $H^1$, so
$\int_{\partial B_n} |\phi_k|^2 \to 0$, yielding a contradiction.
Hence $\psi_k$ is bounded in $H^1$, so that we can assume that $\psi_k \weakto \psi$. Then,
$$  \int_{B_n^c} u_n^{p-1} |\psi_k|^2 \to  \int_{B_n^c} u_n^{p-1} |\psi|^2, \ \  \int_{\partial B_n} |\psi_k|^2 \to \int_{\partial B_n} |\psi|^2.$$
Above we have used the fact that $u_n$ decays to $0$ at infinity and the fact that the embedding $H^1(B_n^c) \hookrightarrow L^2(\partial B_n)$ is compact. We conclude that the convergence is strong and that $\psi$ is a minimizer for $\chi_n$.

\medskip {\bf Step 2:} We pass now to the limit.\\
Let us denote by $\psi_n$  the minimizer of $\chi_n$ renormalized  with respect to the $H^1$ norm. Observe that $\psi_n$ is a solution of the equation:
\begin{equation} \label{linear1}
-\Delta \psi_n + \psi_n - p u_n^{p-1} \psi = \chi_n \psi  \ \  \mbox{in }\; B_n^c. \\
\end{equation}
Moreover,
\begin{equation} \label{papa}\int_{B_n^c} | \nabla \psi_n|^2 + (1-\chi_n) |\psi_n|^2 - p u_n^{p-1} |\psi_n|^2 - \frac{N-1}{R_n} \int_{\partial B_n} |\psi_n|^2 = 0.\end{equation}
By a Cantor diagonal process, $\psi_n \rightharpoonup \psi_0 \in H^1_G(B_r^c)$ for any $r>0$, where $\psi_0 \in H^1(\RN)$ (recall that $H_0^1(\RN \setminus \{0\})=H^1(\RN)$).

\medskip {\bf Step 3:} We show here that
$$  \int_{B_n^c} u_n^{p-1} |\psi_n|^2 \to  \int_{\R^N} U^{p-1} |\psi_0|^2.$$
Indeed, given any $\e >0$, $\psi_n \weakto \psi_0$ in $H^1(B_\e^c)$, which implies that $\psi_n^2 \weakto \psi_0^2$ in $L^{\frac{p+1}{2}}$. Moreover $u_n \to U$ in $H^1(\RN)$, hence:
$$  \int_{B_\e^c} u_n^{p-1} |\psi_n|^2 \to  \int_{B_\e^c} U^{p-1} |\psi_0|^2.$$
Apply now the H\"older inequality:
$$ \int_{B_\e \setminus B_n} u_n^{p-1} |\psi_n|^2 \leq \Big ( \int_{B_\e} u_n^{p+1} \Big )^{\frac{p-1}{p+1}} \Big ( \int_{B_n^c} |\psi_n|^{p+1} \Big )^{\frac{2}{p+1}}.$$
Recall that $u_n \to U$ in $L^{p+1}$ so that
$$ \int_{B_\e} |u_n|^{p+1} \leq \int_{\R^N} |u_n-U|^{p+1} + \int_{B_\e} |U|^{p+1} \leq C \e^N,$$
by choosing sufficiently large $n$. Since $\e$ is arbitrary, we conclude the proof of step 2.

\medskip {\bf Step 4:} We get now the desired contradiction.\\
By Lemma \ref{appendix}, $$\frac{N-1}{R_n} \int_{\partial B_n} |\psi_n|^2 \leq \frac{N-1}{N} \int_{B_n^c} | \nabla \psi_n|^2.$$
This, together with Step 2 and \eqref{papa}, implies that $\psi_0 \neq 0$.
In particular, $$\liminf_{n \to +\infty} \int_{B_n^c}|\psi_n|^2 \geq \int_{B_n^c}|\psi_0|^2>0.$$ Plugging this information in \eqref{papa}, and taking into account Lemma \ref{appendix}, we conclude that $\chi_n$ is bounded. Let us assume that $\chi_n \to \chi_0 \leq 0$.
By \eqref{linear1}, $\psi_0$ is a nontrivial weak solution of the problem:
$$ - \Delta \psi_0 + \psi_0 - p U^{p-1} \psi_0 = \chi_0 \psi_0, \ \mbox{ in } \RN \setminus \{0\}.$$
Since $\psi_0 \in H^1(\RN)$, the singularity is removable and $\psi$ is a weak solution in the whole $\RN$. Since $\psi_0$ is $G$-symmetric, the only possibility is $\psi_0 = k Z$, $k \neq 0$ (see \cite{kwong}).
Observe now that $\int_{B_n^c} \psi_n z_n =0$. By the same arguments as in Step 2, we conclude that
$$ \int_{B_n^c} \psi_n z_n  \to \int_{\RN} \psi_0 Z,$$
which yields the desired contradiction.
\end{proof}

\end{document}